\documentclass[a4paper,11pt]{article}
\usepackage{amsmath,graphicx,amsthm,latexsym,amssymb,natbib}

\newtheorem{theorem}{Theorem}[section]
\newtheorem{theoremm}{Theorem}[subsection]
\newtheorem{fact}[theorem]{Fact}
\newtheorem{factt}[theoremm]{Fact}

\title{The deepest point for distributions \\ in infinite dimensional spaces}
\author{\vspace{0.3in} Anirvan Chakraborty\thanks{Research is partially supported by CSIR SPM Fellowship} ~and Probal Chaudhuri}
\date{}
\begin{document}

\maketitle
\vspace{-0.6in}
\begin{center}
 Theoretical Statistics and Mathematics Unit, \\
 Indian Statistical Institute \\ 
 203, B. T. Road, Kolkata - 700108, INDIA. \\
 emails: anirvan\_r@isical.ac.in, probal@isical.ac.in
\end{center}
\vspace{0.15in}

\begin{abstract}
Identification of the center of a data cloud is one of the basic problems in statistics. One popular choice for such a center is the median, and several versions of median in finite dimensional spaces have been studied in the literature. In particular, medians based on different notions of data depth have been extensively studied by many researchers, who defined median as the point, where the depth function attains its maximum value. In other words, the median is the deepest point in the sample space according to that definition. In this paper, we investigate the deepest point for probability distributions in infinite dimensional spaces. We show that for some well-known depth functions like the band depth and the half-region depth in function spaces, there may not be any meaningful deepest point for many well-known and commonly used probability models. On the other hand, certain modified versions of those depth functions as well as the spatial depth function, which can be defined in any Hilbert space, lead 
to some useful notions of the deepest point with nice geometric and statistical properties. The empirical versions of those deepest points can be conveniently computed for functional data, and we demonstrate this using some simulated and real data sets.
\vspace{0.1in} 
\end{abstract}

\section{Introduction}
\label{1}
For a univariate probability distribution, median is a well-known and popular choice of its center. It has several desirable statistical properties, which include equivariance under monotone transformations, asymptotic consistency under very general conditions and high breakdown point. The concept of median has been extended in several ways for probability distributions in finite dimensional Euclidean spaces (see, e.g., \cite{DC11, Smal90} for some reviews). The median can also be defined as the point in the sample space with the highest depth value with respect to appropriate depth functions (see, e.g., \cite{DG92, LPS99, Smal90, ZS00a}). Due to the recent advances in technology and measurement devices, statisticians frequently have to analyze data for which the number of variables is much larger than the sample sizes. Such data can be conveniently viewed as random observations from probability distributions in infinite dimensional spaces, e.g., the space of real-valued functions defined on an interval. \\ 
\indent It turns out that many medians for finite dimensional probability measures do not extend in any natural and meaningful way into infinite dimensional spaces. On the other hand, an extension of the well-known spatial median (see, e.g., \cite{Brow83}) into general Banach spaces was studied in \cite{Kemp87, Vala84}. There has been some recent work on developing depth functions for probability measures in infinite dimensional spaces. The band depth (see, e.g., \cite{LPR09}) and the half-region depth (see, e.g., \cite{LPR11}) have been defined for data in the space of real-valued continuous functions on an  interval on the real line. The authors of those papers have used these depth functions to find central and extreme curves for several real data sets, and also to construct test procedures based on depth based ranking. Recently, the authors of \cite{SG11} have used the deepest point based on the band depth in the construction of boxplots for functional data. In Section \ref{2}, we shall critically 
investigate the deepest points associated with some of the functional depths. In Section \ref{3}, we shall demonstrate the empirical deepest points using some simulated and real functional data. \\
\indent In the univariate setting, one of the main motivations for considering the median is its robustness against outlying observations. According to the traditional measures of robustness like breakdown point, the median is a more robust estimator of location than the mean. In an interesting paper \cite{Sing93}, the author considered an alternative measure of robustness assuming that all the observations remain bounded. Using that measure, the author of \cite{Sing93} obtained some counter-intuitive results regarding the robustness of the median and the mean (see also \cite{MS97}). In particular, it was shown that the mean may be more robust than the median under certain conditions. Unlike the univariate sample median, many multivariate depth based medians fail to achieve $50\%$ breakdown point (see, e.g., \cite{Smal90} for a review). In fact, in $\mathbb{R}^{d}$ for $d \geq 2$, the half-space median has a breakdown point of $1/3$, while that for the simplicial median is atmost $1/(d+2)$ (see, e.g, \cite{
DC11} for a brief review and relevant references). In Section \ref{4.1}, we will consider the breakdown point of the empirical deepest point based on some functional depths. \\
\indent The strong consistency of the empirical versions of most of the medians for finite dimensional data is well-known in the literature. However, for infinite dimensional data, the situation is much more complex. In some cases, like the medians based on the band depth and the half-region depth, strong consistency has been proved under the assumption that the empirical medians along with the unique population median lie in a fixed equicontinuous set (see \cite{LPR09, LPR11}). For the empirical spatial median in infinite dimensions, the convergence has been shown to hold only in a weak sense, namely, for continuous real linear functions of the estimator (see, e.g., \cite{Cadr01, Gerv08}). Although the author of \cite{Cadr01} proved strong consistency of the empirical spatial median in the norm topology of the Hilbert space $L_{2}(\mathbb{R})$, the result requires extremely strong assumptions on the underlying stochastic process, which include boundedness and differentiability of sample paths. Thus, this 
result is not valid for many important processes including the standard Brownian motions and fractional Brownian motions. Although the author of \cite{Kemp87} considered the spatial median for general Banach spaces, he proved the consistency of its empirical version only in the finite dimensional setting. One of the major reasons for the difficulties in proving consistency results of such estimators in infinite dimensional spaces is the noncompactness of the closed unit ball. As a consequence, some of the standard methods of proof in the finite dimensional setup using Arzela-Ascoli Theorem (see, e.g., the proof as given in \cite{Kemp87}) fail in infinite dimensions. Let us mention here that the strong consistency of an updation based alternative estimator of the spatial median has been proved in separable Hilbert spaces (see \cite{CCZ13}). However, this estimator is very different from the empirical spatial median, and is more suited when data arrive sequentially. In Section \ref{4.2}, we will study the 
asymptotic strong consistency of the empirical deepest points for data lying in infinite dimensional function spaces.

\section{The deepest point}
\label{2}
In this section, we will study the deepest point arising from some of the depth functions that are available for probability distributions in infinite dimensional spaces. The spatial depth of ${\bf x} \in \mathbb{R}^{d}$ with respect to the probability distribution of ${\bf X} \in \mathbb{R}^{d}$ is defined as $SD({\bf x}) = 1 - ||E\{({\bf x} - {\bf X})I({\bf X} \neq {\bf x})/||{\bf x} - {\bf X}||\}||$ (see, e.g., \cite{Serf02}). Recall that the spatial median (see, e.g., \cite{Brow83}), say ${\bf m}_{s}$, of the distribution of ${\bf X}$ is given by the minimizer of the function $E\{||{\bf x} - {\bf X}||\}$ (if $E||{\bf X}|| = \infty$, we can minimize $E\{||{\bf x} - {\bf X}|| - ||{\bf X}||\}$ as suggested in \cite{Kemp87}). The definitions of both of the spatial median and the spatial depth function extend naturally to any Hilbert space, e.g., the function space $L_{2}[0,1]$. It has been proved in \cite{Kemp87} that the spatial median is unique if the Hilbert space is strictly convex, and $F$ is not 
entirely supported on a line (see Theorem 2.17 in that paper). Further, if the distribution of ${\bf X}$ is nonatomic, the unique spatial median is the only point in that Hilbert space, which satisfies $E\{({\bf x} - {\bf X})/||{\bf x} - {\bf X}||\} = {\bf 0}$ (see Theorem 4.14 in \cite{Kemp87}). Summarizing all these, we have the following fact.
\begin{fact}  \label{fact1}
Let ${\bf X}$ be a random element in a strictly convex Hilbert space, and suppose that the probability distribution of ${\bf X}$ is nonatomic and not entirely supported on a line in ${\cal X}$. Then, the unique deepest point associated with the spatial depth function $SD({\bf x})$ is the spatial median of ${\bf X}$, and its spatial depth is $1$.
\end{fact}
Let us also mention here that an alternative definition of spatial depth in $\mathbb{R}^{d}$ was considered by the authors of \cite{VZ00}. They defined the spatial depth of ${\bf x} \in \mathbb{R}^{d}$ as $1 - \inf\{w \geq 0 : \mbox{spatial median of} \ (w\delta_{{\bf x}} + F)/(1+w) = {\bf x}\}$. Here $\delta_{{\bf x}}$ denotes the point mass at ${\bf x}$. It can be shown using the characterization of spatial median given in Theorem 4.14 in \cite{Kemp87} that one gets $1 - \max\{0,1 - SD({\bf x}) - p({\bf x})\}$ as the depth of ${\bf x}$ according to this definition. Here $p({\bf x})$ denotes the mass at ${\bf x}$ under the distribution of ${\bf X}$. So, this  definition of spatial depth coincides with the previous definition if $F$ is nonatomic. \\
\indent Infinite dimensional data are often obtained as realizations of functions, e.g., spectrometric data, electrocardiogram records of patients, stock price data, various meteorological data like temperature, rainfall etc. In that case, one may want to exploit the functional nature of the data in constructing depth functions. As mentioned in Section \ref{1}, the band depth and the half-region depth were introduced in the literature for such data. The band depth (BD) of ${\bf x} = \{x_{t}\}_{t \in [0,1]} \in C[0,1]$ relative to the distribution of a random element ${\bf X} = \{X_{t}\}_{t \in [0,1]} \in C[0,1]$ is given by 
\begin{eqnarray*}
BD({\bf x}) = \sum_{j=2}^{J} P\left(\min_{1 \leq i \leq j} X_{i,t} \leq x_{t} \leq \max_{1 \leq i \leq j} X_{i,t} \ \forall \ t \in [0,1]\right),
\end{eqnarray*}
where for $1 \leq i \leq J$, ${\bf X}_{i} = \{X_{i,t}\}_{t \in [0,1]}$ are independent copies of ${\bf X}$. The half-region depth (HRD) of ${\bf x}$ is given by 
\begin{eqnarray*}
HRD({\bf x}) = \min \left\{P\left(X_{t} \leq x_{t} \ \forall \ t \in [0,1]\right), \ P\left(X_{t} \geq x_{t} \ \forall \ t \in [0,1]\right)\right\}.
\end{eqnarray*}
\indent However, as we shall see in the next theorem, for many well-known and commonly used stochastic models for functional data, there is no meaningful notion of deepest points associated with these two  depth functions. These stochastic models include Feller processes in $C[0,1]$. Feller processes are a class of strong Markov processes, whose transition probability function satisfies certain continuity properties (see, e.g., \cite{RY91} for a detailed discussion). Important examples of Feller processes include Brownian motions, Brownian bridges etc. Feller processes have been used for modelling data in physical and biological sciences (see, e.g., \cite{Bott10} for a review and related references).
\begin{theorem} \label{thm1}
Let ${\bf X}$ be a Feller process in $C[0,1]$ starting at $x_{0} \in \mathbb{R}$, which is symmetric about ${\bf a} = \{a_{t}\}_{t \in [0,1]} \in C[0,1]$, i.e., the distribution of ${\bf X} - {\bf a}$ is same as that of ${\bf a} - {\bf X}$. Also, assume that $X_{t}$ has a continuous distribution for all $t \in (0,1]$ with any finite dimensional marginal $(X_{t_{1}},X_{t_{2}},\ldots,X_{t_{d}})$ of ${\bf X}$ having a positive density in a neighbourhood of its center of symmetry $(a_{t_{1}},a_{t_{2}},\ldots,a_{t_{d}})$, where $0 < t_{1} < t_{2} < \ldots < t_{d} \leq 1$. Then, the band depth vanishes identically on $C[0,1]$. Moreover, the half-region depth of ${\bf a}$ is zero.
\end{theorem}
\begin{proof}[Proof of Theorem \ref{thm1}]
Since $X_{0} = x_{0}$ with probability one, if ${\bf y} = \{y_{t}\}_{t \in [0,1]} \in C[0,1]$ is such that $y_{0} \neq x_{0}$, then ${\bf y}$ cannot be contained in any band formed by the ${\bf X}_{i}$'s. Consequently, $BD({\bf y}) = 0$. So, it is enough to prove the result for any ${\bf y} = \{y_{t}\}_{t \in [0,1]} \in C[0,1]$ satisfying $y_{0} = x_{0}$. Let ${\bf y}_{m} = (y_{1/2^{m}},y_{2/2^{m}},\ldots,y_{2^{m}/2^{m}})$. It follows from part 2 of Theorem 1 in \cite{LPR09} that $BD_{m}({\bf y}_{m}) \leq BD_{m}({\bf a}_{m})$ for all $m \geq 1$, where the function $BD_{m}$ is the $m$-dimensional band depth calculated using the distribution of ${\bf X}_{m} = (X_{1/2^{m}},X_{2/2^{m}},\ldots,X_{2^{m}/2^{m}})$. Since
\begin{eqnarray*}
&& \{\min_{1 \leq i \leq j} X_{i,k/2^{m+1}} \leq x_{k/2^{m+1}} \leq \max_{1 \leq i \leq j} X_{i,k/2^{m+1}} \ \forall \ k = 1,\ldots,2^{m+1}\} \\
&\subseteq& \{\min_{1 \leq i \leq j} X_{i,k/2^{m}} \leq x_{k/2^{m}} \leq \max_{1 \leq i \leq j} X_{i,k/2^{m}} \ \forall \ k = 1,\ldots,2^{m}\} 
\end{eqnarray*}
for all $1 \leq j \leq J$, it follows that $BD_{m+1}({\bf y}_{m+1}) \leq BD_{m}({\bf y}_{m})$ for all $m \geq 1$. So,
\begin{eqnarray*}
&& \lim_{m \rightarrow \infty} BD_{m}({\bf y}_{m}) \\
&=& \sum_{j=2}^{J} P\left(\min_{1 \leq i \leq j} X_{i,k/2^{m}} \leq x_{k/2^{m}} \leq \max_{1 \leq i \leq j} X_{i,k/2^{m}} \ \forall \ k = 1,\ldots,2^{m} \ \mbox{and} \ m \geq 1\right)
\end{eqnarray*}
Then, using the almost sure uniform continuity of the sample paths, we have
\begin{eqnarray*}
BD({\bf y}) = \lim_{m \rightarrow \infty} BD_{m}({\bf y}_{m}) \leq \lim_{m \rightarrow \infty} BD_{m}({\bf a}_{m}) = BD({\bf a}).
\end{eqnarray*}
The proof will be complete if we show that $BD({\bf a}) = 0$. Let us now consider the multivariate Feller process $\{(X_{1,t},X_{2,t},\ldots,X_{j,t})\}_{t \in [0,1]}$ for $1 \leq j \leq J$. Define ${\bf Y} = \{Y_{t}\}_{t \in [0,1]} = {\bf X} - {\bf a}$. Since $\{X_{t}\}_{t \in [0,1]}$ is a Feller process starting at $x_{0} = a_{0}$ and symmetric about ${\bf a}$, ${\bf Y}$ is a Feller process starting at $0$ and symmetric about ${\bf 0}$. Let $T_{j} = \inf\{t > 0 : \min_{1 \leq i \leq j} X_{i,t} > a_{t}\} = \inf\{t > 0 : \min_{1 \leq i \leq j} Y_{i,t} > 0\}$ and $S_{j} = \inf\{t > 0 : \max_{1 \leq i \leq j} X_{i,t} < a_{t}\} = \inf\{t > 0 : \max_{1 \leq i \leq j} Y_{i,t} < 0\}$. From the continuity of the sample paths and using Propositions 2.16 and 2.17 in \cite{RY91}, we get that $P(T_{j} = 0) = 0$ or $1$ and $P(S_{j} = 0) = 0$ or $1$ for all $1 \leq j \leq J$. Since $P(T_{j} = 0) = \lim_{t \downarrow 0} P(T_{j} \leq t) \geq 2^{-j}$ and $P(S_{j} = 0) = \lim_{t \downarrow 0} P(S_{j} \leq t) \geq 2^{-j}$, we 
have $P(T_{j} = 0) = P(S_{j} = 0) = 1$ for all $1 \leq j \leq J$.  \\
The proof of the fact that the half-region depth of ${\bf a}$ is $0$ follows from the above arguments after taking $J = 1$. 
\end{proof}
Since any reasonable depth function should assign maximum value to the point of symmetry of a distribution, and it should preferably be the unique maximizer, it follows from the above theorem that neither of these two depth functions yields any useful notion of deepest point for such stochastic processes. It will be appropriate to note here that it was proved in \cite{CC12} that both of the band depth and the half-region depth vanish on a set of measure one for a class of Feller processes in $C[0,1]$. The assumptions of symmetry and the existence of densities of all the finite dimensional marginals of ${\bf X}$ yield a stronger result for the band depth here. \\
\indent The previous theorem provides a mathematical explanation for an observation made in \cite{LPR09, LPR11} that both of the band depth and the half-region depth tend to take small values if the sample paths cross each other often. This led the authors of those papers to consider modified versions of these depth functions, called modified band depth (MBD) and modified half-region depth (MHRD), respectively. The definitions of both of these depths directly extend to random elements in the space, say ${\cal F}(I)$, of real-valued functions defined on a set $I$. Here, for an appropriate $\sigma$-field ${\cal I}$ and a probability measure $\Lambda$ on $I$, we consider the probability space $(I,{\cal I},\Lambda)$. For example, $I$ can be $[0,1]$ with ${\cal I}$ as its Borel $\sigma$-field and $\Lambda$ as the Lebesgue measure on $[0,1]$. The modified band depth of ${\bf x} \in {\cal F}(I)$ relative to the distribution of ${\bf X} = \{X_{{\bf t}}\}_{{\bf t} \in I} \in {\cal F}(I)$ is defined as 
\begin{eqnarray*}
MBD({\bf x}) = \sum_{j=2}^{J} E\left[\int_{I} I\left(\min_{i=1,\ldots,j} X_{i,{\bf t}} \leq x_{{\bf t}} \leq \max_{i=1,\ldots,j} X_{i,{\bf t}}\right) \Lambda(d{\bf t})\right], 
\end{eqnarray*}
where ${\bf X}_{i} = \{X_{i,{\bf t}}\}_{{\bf t} \in I}$ are $J$ independent copies of ${\bf X}$. A closely related notion of depth, called integrated data depth (IDD) \cite{FM01}, is defined on ${\cal F}(I)$ as $IDD({\bf x}) = \int_{I} D_{{\bf t}}(x_{{\bf t}}) \Lambda(d{\bf t})$, where $D_{{\bf  t}}$'s are depth functions on the real line. Henceforth, we shall assume that for each ${\bf t} \in I$, $D_{{\bf t}}$ is maximized at the median of $X_{{\bf t}}$ if $X_{{\bf t}}$ has a continuous distribution. The modified half-region depth of ${\bf x}$ relative to the distribution of ${\bf X}$ is given by
\begin{eqnarray*}
MHRD({\bf x}) = \min\left\{E\left[\int_{I} I(X_{{\bf t}} \leq x_{{\bf t}}) \Lambda(d{\bf t})\right], \ E\left[\int_{I} I(X_{{\bf t}} \geq x_{{\bf t}}) \Lambda(d{\bf t})\right]\right\}.
\end{eqnarray*}
The next result gives a description of points, which maximize the above three depth functions. For this, let us write ${\bf m}_{c} = \{m_{{\bf t}}\}_{{\bf t} \in I}$, where $m_{{\bf t}}$ denotes the median of $X_{{\bf t}}$ for ${\bf t} \in I$.  
\begin{theorem} \label{thm2}
Let ${\bf X}$ be a random element in ${\cal F}(I)$ such that $X_{{\bf t}}$ has a continuous distribution, which is strictly increasing in a neighbourhood of $m_{{\bf t}}$ for each ${\bf t} \in I$. Then, ${\bf m}_{c}$ is a maximizer of the modified band depth, the modified half-region depth and the integrated data depth. Any ${\bf m}^{*}_{c}$, which equals ${\bf m}_{c}$ for ${\bf t} \in I$ except on a set of $\Lambda$-measure zero, is also a maximizer of the modified band depth and the integrated data depth. Moreover, any ${\bf m}^{**}_{c} = \{m^{**}_{{\bf t}}\}_{{\bf t} \in I}$ satisfying $\int_{I} F_{{\bf t}}(m^{**}_{{\bf t}}) \Lambda(d{\bf t}) = 1/2$ is a maximizer of the modified half-region depth.
\end{theorem}
It is clear that ${\bf m}^{**}_{c}$, which satisfies $\int_{I} F_{{\bf t}}(m^{**}_{{\bf t}}) \Lambda(d{\bf t}) = 1/2$, may differ from ${\bf m}_{c}$ on a set of positive $\Lambda$-measure unlike ${\bf m}^{*}_{c}$. Moreover, although ${\bf m}^{**}_{{\bf c}}$ is a maximizer of the modified half-region depth, its components $m^{**}_{{\bf t}}$ may be far from being univariate medians. For instance, let us consider a standard Brownian motion ${\bf X} = \{X_{t}\}_{t \in [0,1]}$. Define the function ${\bf f} = \{f_{t}\}_{t \in [0,1]}$, where $f_{t}$ is the $\alpha^{th}$ percentile of $X_{t}$ for $t \in [0,1/2)$, and $f_{t}$ is the $(1 - \alpha)^{th}$ percentile of $X_{t}$ for $t \in [1/2,1]$. Here $\alpha \in (0,1)$ can be as small or as large as we like. Then, any such ${\bf f}$ is a maximizer of the modified half-region depth.
\begin{proof}[Proof of Theorem \ref{thm2}]
From the definition of modified band depth we have
\begin{eqnarray}
MBD({\bf x}) &=& \sum_{j=2}^{J} E\left[\int_{I} I\left(\min_{i=1,\ldots,j} X_{i,{\bf t}} \leq x_{{\bf t}} \leq \max_{i=1,\ldots,j} X_{i,{\bf t}}\right) \Lambda(d{\bf t})\right]    \nonumber \\
&=& \sum_{j=2}^{J} \left[\int_{I} P\left(\min_{i=1,\ldots,j} X_{i,{\bf t}} \leq x_{{\bf t}} \leq \max_{i=1,\ldots,j} X_{i,{\bf t}}\right) \Lambda(d{\bf t})\right]  \nonumber \\
&=& \sum_{j=2}^{J} \int_{I} \left[1 - F_{{\bf t}}^{j}(x_{{\bf t}}) - (1 - F_{{\bf t}}(x_{{\bf t}}))^{j}\right] \Lambda(d{\bf  t}),  \label{eq2.3.1} 
\end{eqnarray}
where the second equality in the above follows from Fubini's theorem. Here $F_{{\bf t}}$ denotes the distribution of $X_{{\bf t}}$ for ${\bf t} \in I$. For each ${\bf t} \in I$, the integrand in (\ref{eq2.3.1}) is maximized iff $F_{{\bf t}}(x_{{\bf t}}) = 1/2$, which can be easily verified using standard calculus. This implies that the term in the right hand side of equation (\ref{eq2.3.1}) is maximized iff $F_{{\bf t}}(x_{{\bf t}}) = 1/2$ for all ${\bf t} \in I$, except perhaps on a subset of $I$ with $\Lambda$-measure zero. Hence, the modified band depth is maximized at ${\bf m}_{c}$, and also at any ${\bf m}^{*}_{c}$, which equals ${\bf m}_{c}$ outside a $\Lambda$-null set. \\
\indent Since $D_{{\bf t}}$ is maximized at $m_{{\bf t}}$ for each ${\bf t} \in I$, it follows that the integrated data depth function $\int_{I} D_{{\bf t}}(x_{{\bf t}}) \Lambda(d{\bf t})$ is maximized at ${\bf m}_{c}$, or at any ${\bf m}^{*}_{c}$, which equals ${\bf m}_{c}$ except on a $\Lambda$-null set. \\
\indent For the modified half-region depth, we have using Fubini's theorem
\begin{eqnarray}
MHRD({\bf x}) &=& \min\left\{E\left[\int_{I} I(X_{{\bf t}} \leq x_{{\bf t}}) \Lambda(d{\bf t})\right], \ E\left[\int_{I} I(X_{{\bf t}} \geq x_{{\bf t}}) \Lambda(d{\bf t})\right]\right\}    \nonumber \\
&=& \min\left\{\left[\int_{I} P(X_{{\bf t}} \leq x_{{\bf t}}) \Lambda(d{\bf t})\right], \ \left[\int_{I} P(X_{{\bf t}} \geq x_{{\bf t}}) \Lambda(d{\bf t})\right]\right\}    \nonumber \\
&=& \min\left\{\int_{I} F_{{\bf t}}(x_{{\bf t}}) \Lambda(d{\bf t}), \ 1 - \int_{I} F_{{\bf t}}(x_{{\bf t}}) \Lambda(d{\bf t})\right\}.   \label{eq2.3.2}
\end{eqnarray}
The maximum value of the right hand side of (\ref{eq2.3.2}) is $1/2$. Since $F_{{\bf t}}(m_{{\bf t}}) = 1/2$ for all ${\bf t} \in I$, ${\bf m}_{c}$ is a maximizer of the modified half-region depth. Further, any ${\bf m}^{**}_{c} = \{m^{**}_{{\bf t}}\}_{{\bf t} \in I}$ satisfying $\int_{I} F_{{\bf t}}(m^{**}_{{\bf t}}) \Lambda(d{\bf t}) = 1/2$ will also maximize the modified half-region depth.
\end{proof}

\section{Computation and data analytic demonstration of the empirical deepest point}
\label{3}
We shall now consider empirical versions of the two deepest points discussed in the previous section. Suppose that ${\bf X}_{i}$, $1 \leq i \leq n$, are i.i.d observations from a distribution in a Hilbert space, e.g., the function space $L_{2}[0,1]$. The empirical version of ${\bf m}_{s}$, say $\widehat{{\bf m}}_{s}$, is given by the minimizer of the function $n^{-1} \sum_{i=1}^{n} ||{\bf x} - {\bf X}_{i}||$, where $||.||$ is the norm in that Hilbert space. For the empirical version of ${\bf m}_{c}$, suppose that ${\bf X}_{i} = \{X_{i,{\bf t}}\}_{{\bf t} \in I}$, $1 \leq i \leq n$, are i.i.d. observations from a probability distribution in the space ${\cal F}(I)$. Then, the natural estimator of ${\bf m}_{c}$ is $\widehat{{\bf m}}_{c} = \{\widehat{m}_{c,{\bf t}}\}_{{\bf t} \in I}$, where $\widehat{m}_{c,{\bf t}}$ is the median corresponding to the empirical distribution of $X_{{\bf t}}$. Throughout this paper, we shall use the conventional definition of median, i.e., if $n$ is even, $\widehat{m}_{{\bf t}} = (
X_{(n/2),{\bf t}} + X_{(n/2 + 1),{\bf t}})/2$, and if $n$ is odd, $\widehat{m}_{{\bf t}} = X_{((n+1)/2),{\bf t}}$. The empirical coordinatewise median $\widehat{{\bf m}}_{c}$ is easy to compute. Although there are some algorithms for computing the empirical spatial median $\widehat{{\bf m}}_{s}$ for data in finite dimensional spaces (see, e.g., \cite{BZ79, VZ00}), some of them, which involve the inverse of the Hessian matrix of the function being optimized, cannot be used in the infinite dimensional setting. This is due to the noninvertibility of a Hessian matrix, which occurs when the data dimension exceeds the sample size. It is known that the empirical spatial median always lies in the closed convex hull generated by the sample points (see, e.g., Remark 4.20 in \cite{Kemp87}). This property of the empirical spatial median has been used in \cite{Gerv08} to reduce the infinite dimensional minimization problem associated with the empirical spatial median to a convex minimization over the compact set $\{\sum_{k=1}^{n} \alpha_{k} = 1 : \alpha_{k} \in [0,1]\}$, when there are $n$ data points.  \\
\indent We now demonstrate the empirical deepest points using some simulated and real data sets. For both the simulated and the real data sets, the empirical spatial median is computed using the iterative algorithm proposed in \cite{Gerv08}. Each of the simulated data sets that we consider has $10$ observations. One set is generated from the standard Brownian motion in $[0,1]$, and the two other sets are generated from fractional Brownian motions in $[0,1]$ with Hurst indices $H = 0.3$ and $H = 0.7$. Recall that a fractional Brownian motion with Hurst index $H \in (0,1)$ is a zero mean Gaussian process having covariance kernel $K(t,s) = 0.5(t^{2H} + s^{2H} - |t-s|^{2H})$ for $t,s \in [0,1]$. The samples are observed at $101$ equispaced points in $[0,1]$. Figures 1, 2 and 3 show the plots of the sample curves along with the empirical coordinatewise median and the empirical spatial median. It is observed from the plots that both the empirical medians are close to the zero function, which is the center of 
symmetry for all three distributions. \\
\indent The real data considered here is the growth acceleration data set, which is derived from the Berkeley growth data. The latter is available in the R package ``fda'' (see \texttt{http://rss.ac s.unt.edu/Rdoc/library/fda/html/growth.html}) and contains two subclasses, namely, the boys and the girls. Heights of $39$ boys and $54$ girls were measured at $31$ time points between ages $1$ and $18$ years. The growth acceleration curves are obtained through monotone spline smoothing available in the R package ``fda'', and those are recorded at $101$ equispaced ages in the interval $[1,18]$ years. Figure \ref{fig4} shows the plots of the acceleration data of the boys and the girls along with the empirical coordinatewise medians and the empirical spatial medians. It is seen from Figure \ref{fig4} that both of the empirical medians are close to the central curves in each data set.

\begin{figure}[htb]
\begin{center}
\includegraphics[height=5.5in,width=2.7in,angle=270]{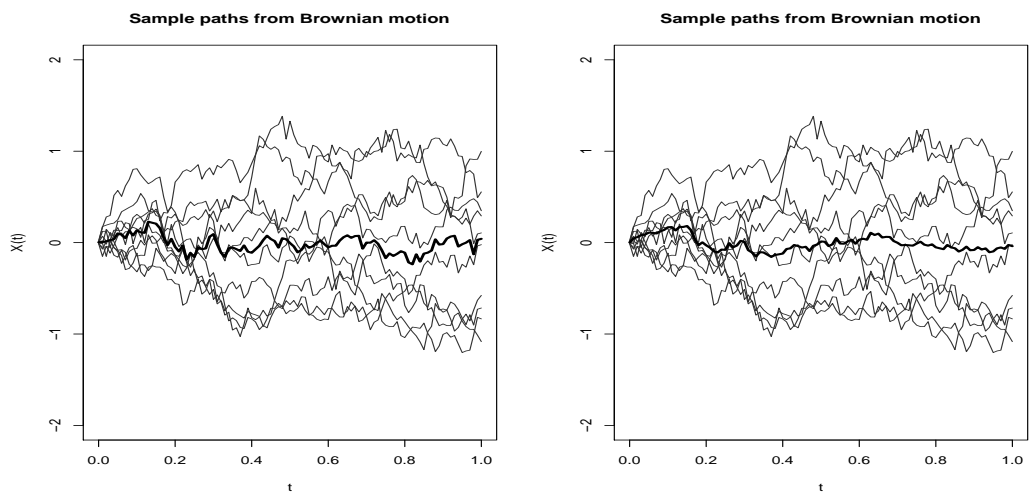} 
\end{center}
\caption{{\it Plots of sample paths from standard Brownian motion along with the empirical coordinatewise median (bold black curve in the left panel) and the empirical spatial median (bold black curve in the right panel).}}
\label{fig1}
\end{figure}

\begin{figure}
\begin{center}
\includegraphics[height=5.5in,width=2.7in,angle=270]{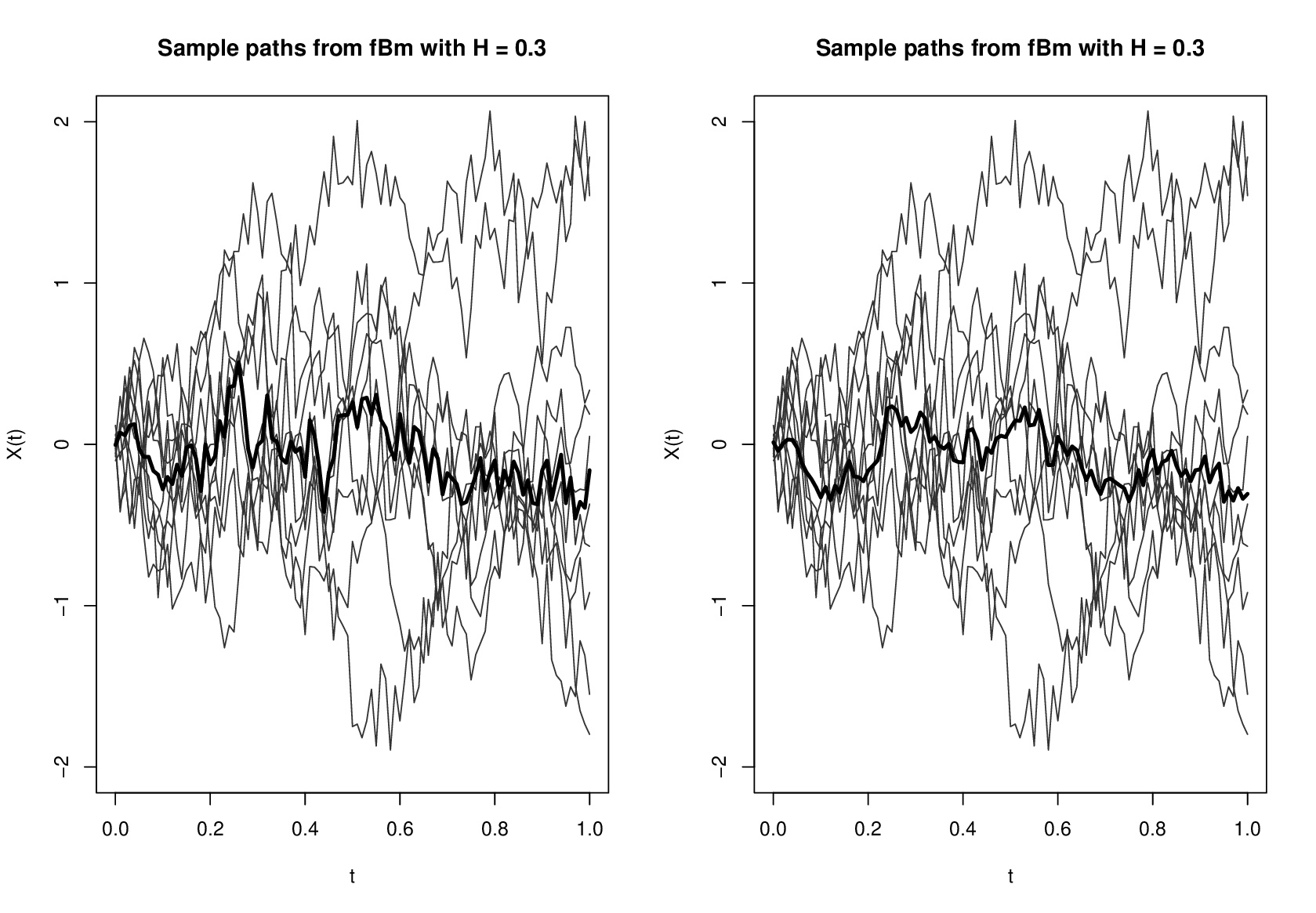} 
\end{center}
\caption{{\it Plots of sample paths from fractional Brownian motion with Hurst index $H = 0.3$ along with the empirical coordinatewise median (bold black curve in the left panel) and the empirical spatial median (bold black curve in the right panel).}}
\label{fig2}
\end{figure}

\begin{figure}
\begin{center}
\includegraphics[height=5.5in,width=2.7in,angle=270]{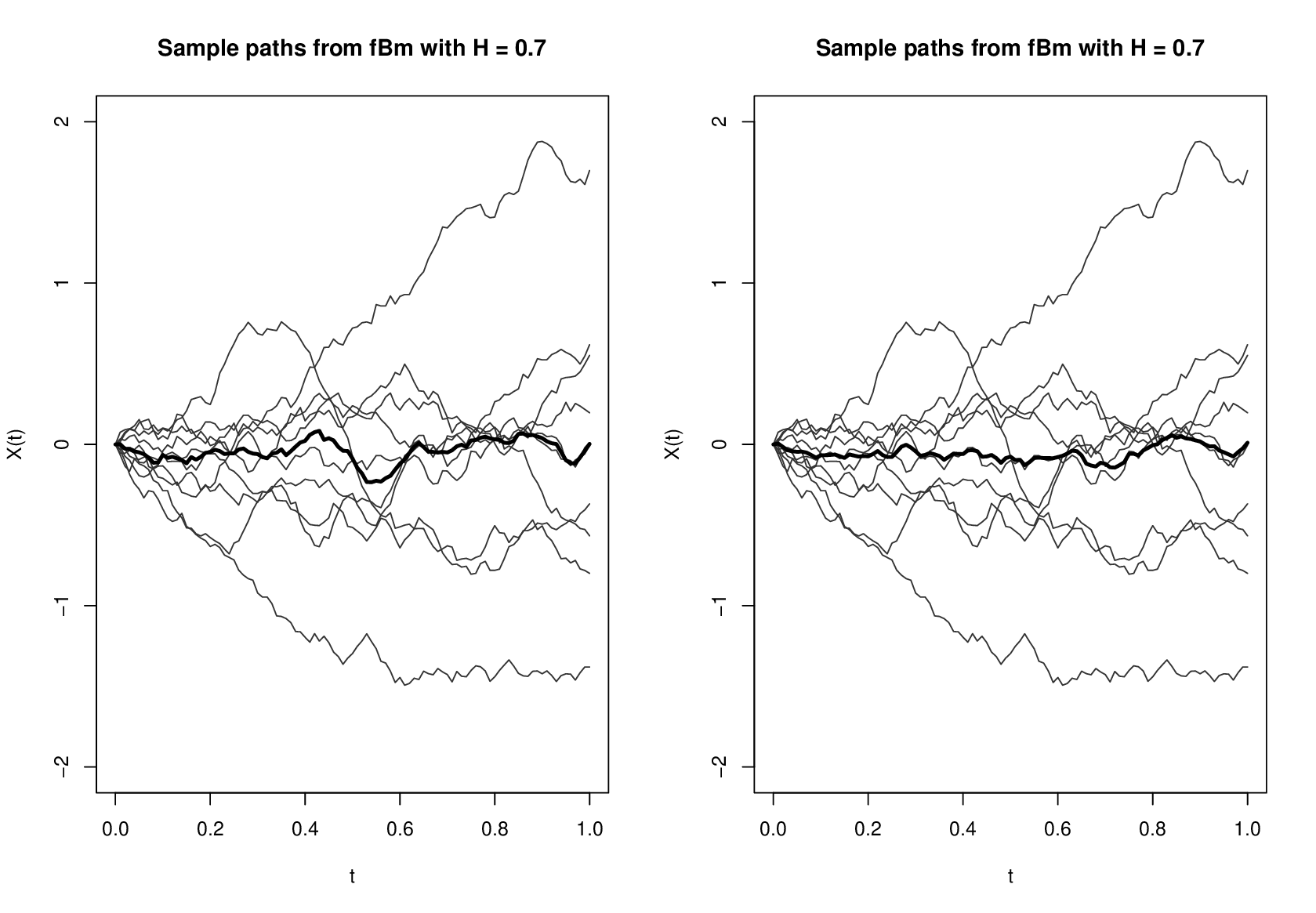} 
\end{center}
\caption{{\it Plots of sample paths from fractional Brownian motion with Hurst index $H = 0.7$ along with the empirical coordinatewise median (bold black curve in the left panel) and the empirical spatial median (bold black curve in the right panel).}}
\label{fig3}
\end{figure}

\begin{figure}
\begin{center}
\includegraphics[height=5.5in,width=3in,angle=270]{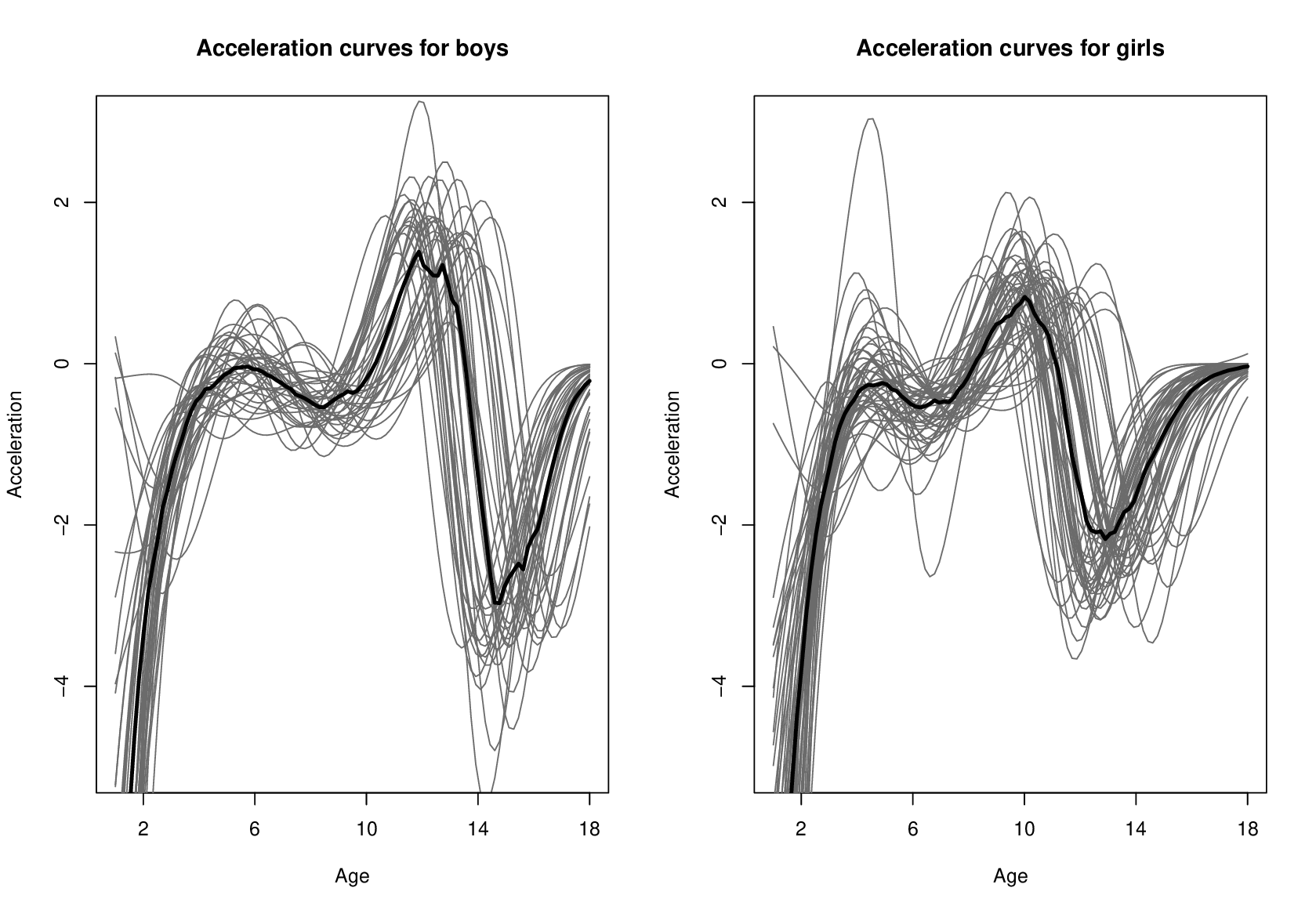} 
\includegraphics[height=5.5in,width=3in,angle=270]{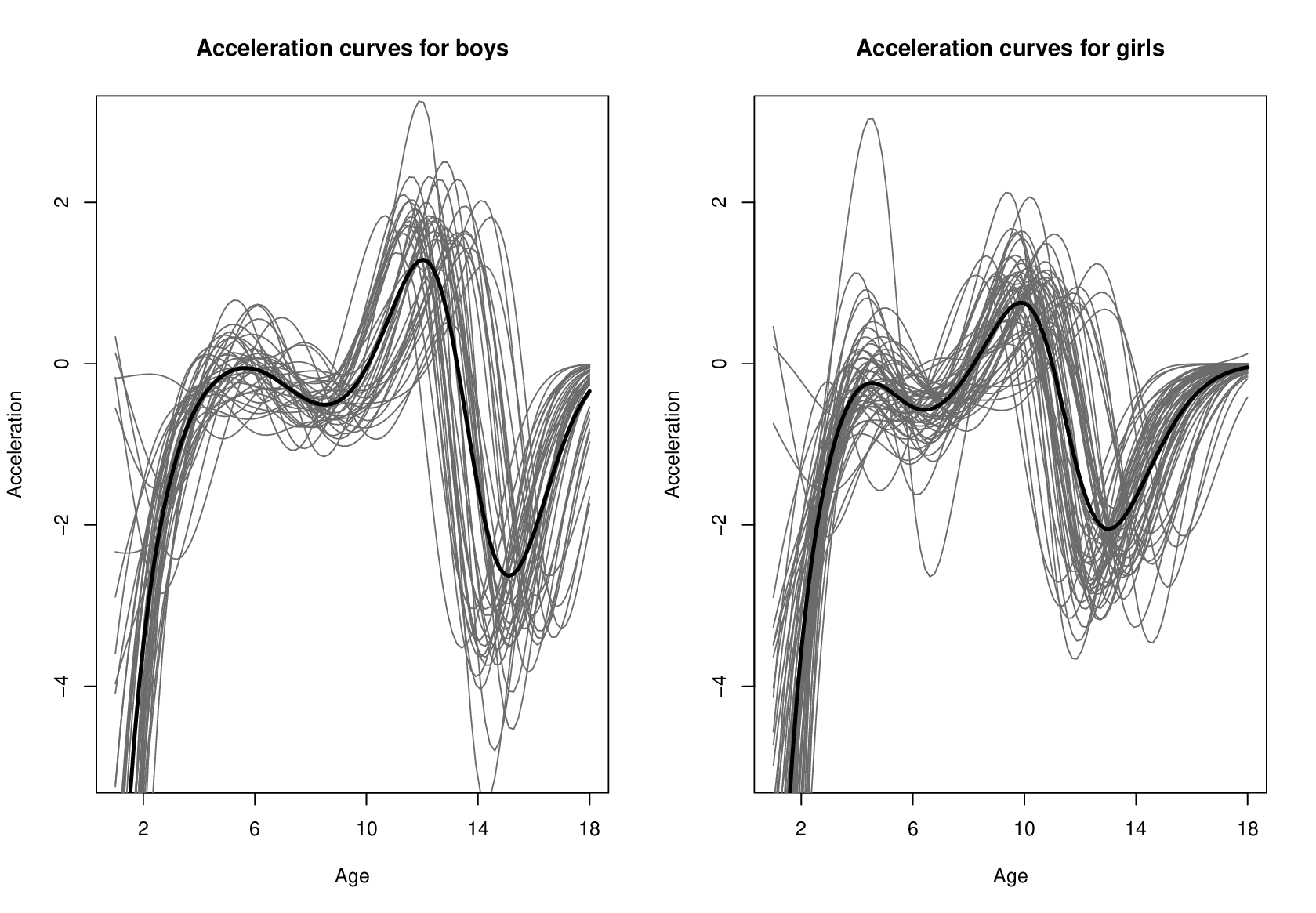}
\end{center}
\caption{{\it Plots of acceleration curves for boys (left panel) and girls (right panel) along with the corresponding empirical coordinatewise medians (bold black curves in the top panel) and the empirical spatial medians (bold black curves in the bottom panel).}}
\label{fig4}
\end{figure}

\section{Geometric and statistical properties of the empirical deepest point}
\label{4}
We shall now discuss the properties of the empirical deepest point mentioned in the previous section. The empirical spatial median is unique if the data points are not concentrated on a line, which follows from Theorem 2.17 in \cite{Kemp87}. Further, if the empirical spatial median is not one of the data points, then using Theorem 4.14 in \cite{Kemp87} it can be verified that the empirical spatial depth of $\widehat{{\bf m}}_{s}$ is $1$, where the empirical spatial depth of ${\bf x}$ is given by $1 - n^{-1} \left\|\sum_{i=1}^{n} \left({\bf x} - {\bf X}_{i}\right)I\left({\bf X}_{i} \neq {\bf x}\right)/\left\|{\bf x} - {\bf X}\right\| \right\|$. In other words, in such situations, the empirical spatial median $\widehat{{\bf m}}_{s}$ is the unique maximizer of the empirical spatial depth.   \\
\indent Let $\widehat{F}_{{\bf t}}$ denote the empirical distribution of $X_{{\bf t}}$ for each ${\bf t} \in I$. Then, the empirical modified band depth of ${\bf x} = \{x_{{\bf t}}\}_{{\bf t} \in I}$ is given by $\sum_{j=2}^{J} \int_{I} \left[1 - \widehat{F}_{{\bf t}}^{j}(x_{{\bf t}}-) - (1 - \widehat{F}_{{\bf t}}(x_{{\bf t}}))^{j}\right] \Lambda(d{\bf t})$. Also, the empirical integrated data depth of ${\bf x}$ is given by $\int_{I} \widehat{D}_{{\bf t}}(x_{{\bf t}}) \Lambda(d{\bf t})$, where $\widehat{D}_{{\bf t}}(x_{{\bf t}})$'s denote the empirical versions of the depths $D_{{\bf t}}$'s. We can assume that $\widehat{D}_{{\bf t}}(x_{{\bf t}})$ is maximized at $\widehat{m}_{{\bf t}}$ for all ${\bf t} \in I$, which is true for almost any depth function for univariate data. The empirical half-region depth of ${\bf x}$ is given by $\min\left\{\int_{I} \widehat{F}_{{\bf t}}(x_{{\bf t}}) \Lambda(d{\bf t}), \ 1 - \int_{I} \widehat{F}_{{\bf t}}(x_{{\bf t}}-) \Lambda(d{\bf t})\right\}$. It can be verified that $\widehat{{\bf m}}_{c}$ is a maximizer of the empirical versions of all of the three depth functions mentioned above if the distribution of $X_{{\bf t}}$ is continuous for all ${\bf t} \in I$. Further, any $\widehat{{\bf m}}^{*}_{c}$, which differ from $\widehat{{\bf m}}_{c}$ only on a $\Lambda$-null set is also a maximizer of the empirical modified band depth and the empirical integrated data depth. Also, there is no unique maximizer of the empirical modified half-region depth, and any $\widehat{{\bf m}}^{**}_{c} = \{\widehat{m}^{**}_{{\bf t}}\}_{{\bf t} \in I}$ satisfying $\int_{I} \widehat{F}_{{\bf t}}(\widehat{m}^{**}_{{\bf t}}) = 1/2$ will be a maximizer of the empirical modified half-region depth like its population counterpart.  \\
\indent It is easy to see that the empirical coordinatewise median is equivariant under any coordinatewise monotone transformation given by ${\bf x} \mapsto \Psi({\bf x})$, where ${\bf x} = \{x_{{\bf t}}\}_{{\bf t} \in I}$, $\Psi({\bf x}) = \{\psi_{{\bf t}}(x_{{\bf t}})\}_{{\bf t} \in I}$, and $\psi_{{\bf t}}$ is a monotone function for each ${\bf t} \in I$. These include location shifts ${\bf x} \mapsto {\bf x} + {\bf c}$, where ${\bf c} \in {\cal F}(I)$ as well as coordinatewise scale transformations ${\bf x} \mapsto {\bf x}_{1}$, where ${\bf x}_{1} = \{a_{{\bf t}}x_{{\bf t}}\}_{{\bf t} \in I}$ and $a_{{\bf t}} > 0$ for each ${\bf t} \in I$. On the other hand, the empirical spatial median is equivariant under location shifts and homogeneous scale transformations ${\bf x} \mapsto a{\bf x}$, where $a > 0$. Moreover, it is also equivariant under linear isometries, i.e., any linear map ${\bf T} : {\cal X} \rightarrow {\cal X}$ such that $||{\bf T}{\bf x}|| = ||{\bf x}||$ for all ${\bf x}$ in the Hilbert space 
${\cal X}$.

\subsection{Robustness properties of the empirical deepest point}
\label{4.1}
It was mentioned in Section \ref{1} that many finite dimensional depth based medians, e.g., the half-space median, the simplicial median, do not have $50\%$ breakdown point unlike the univariate median. We shall now discuss the robustness properties of the two empirical deepest points discussed in Section \ref{3} in terms of their breakdown points. Let us first consider the breakdown point of $\widehat{{\bf m}}_{c}$, and for this we assume that $\int_{I} |X_{{\bf t}}| \Lambda(d{\bf t}) < \infty$ with probability $1$. This assumption holds if $I$ is a compact set, and with probability one, the process ${\bf X}$ has continuous paths(e.g., Brownian motions on a compact interval). Then, it can be shown that $\widehat{{\bf m}}_{c}$ is a minimizer of the function $\int_{I} \sum_{i=1}^{n} |X_{i,{\bf t}} - x_{{\bf t}}| \Lambda(d{\bf t})$ over the set of ${\bf x} \in {\cal F}(I)$, which satisfies $\int_{I} |x_{{\bf t}}| \Lambda(d{\bf t}) < \infty$. Here ${\bf X}_{i} = \{X_{i,{\bf t}}\}_{{\bf t} \in I}$ are the sample observations. In other words, $\widehat{{\bf m}}_{c}$ is an empirical spatial median of the distribution of ${\bf X}$ in the Banach space of real-valued absolutely integrable functions defined on $I$. Hence, from Theorem 2.10 in \cite{Kemp87}, we get that $\widehat{{\bf m}}_{c}$ has $50\%$ breakdown point. The same theorem also implies that the empirical spatial median has $50\%$ breakdown point. The following result summarizes the above observations.
\begin{factt} \label{fact4}
The empirical deepest point $\widehat{{\bf m}}_{s}$ has $50\%$ breakdown point. Further, if the random element ${\bf X} \in {\cal F}(I)$ satisfies $\int_{I} |X_{{\bf t}}| \Lambda(d{\bf t}) < \infty$ with probability $1$, then $\widehat{{\bf m}}_{c}$ also has $50\%$ breakdown point.
\end{factt}

\subsection{Asymptotic consistency of empirical deepest points}
\label{4.2}
We shall now investigate the strong consistency of the empirical deepest points. The following result asserts uniform strong consistency of the empirical coordinatewise median, where the uniformity is over a subset of $I$ the size of which grows with the sample size at an appropriate rate.
\begin{factt}   \label{fact2}
Suppose that $X_{{\bf t}}$ has a density $f_{{\bf t}}$ in a neighbourhood of $m_{{\bf t}}$ for each ${\bf t} \in I$. Assume that for some $c_{0}, \eta_{0} > 0$, we have $\inf_{|x-m_{{\bf t}_{k}}| < \eta_{0}} f_{{\bf t}_{k}}(x) > c_{0}$ for all $1 \leq k \leq d_{n}$ and $d_{n} \geq 1$. Then, if $log(d_{n}) = o(n)$ as $n \rightarrow \infty$, we have $\sup_{1 \leq k \leq d_{n}} |\widehat{m}_{{\bf t}_{k}} - m_{{\bf t}_{k}}| \rightarrow 0$ as $n \rightarrow \infty$ almost surely.
\end{factt}
The proof of the above result can be obtained using the arguments in the proof of Corollary 6 in \cite{KM07}. The authors of \cite{KM07} considered the coordinatewise median for high dimensional data when the dimension increases with the sample size. Using sharp uniform bounds on the marginal empirical processes corresponding to the coordinate variables, they obtained the rate of convergence of $\widehat{{\bf m}}_{c}$ under the same set of assumptions used in Fact \ref{fact2} (see Corollary 6 in \cite{KM07}). \\
\indent In practice, the index set $I$ is most often a subset of $\mathbb{R}^{p}$ for some fixed $p \geq 1$, and the process ${\bf X}$ is observed at $d_{n}$ grid points in $I$. We can then construct a functional estimator of ${\bf m}_{c}$ from $\widehat{{\bf m}}_{c}$ as follows. For any point ${\bf t}$ in $I$, which is not a grid point, we can define $\widehat{m}_{{\bf t}}$ by the average of the empirical medians corresponding to its $k$ nearest grid points. Here $k \geq 1$ is a fixed integer. In that case, the uniform consistency of this functional estimator over the whole of $I$ can be derived from Fact \ref{fact2}. For this derivation, let us assume that $I$ is compact, and the grid points become dense in $I$ as $n \rightarrow \infty$. In other words, for each ${\bf t} \in I$, any fixed neighbourhood of ${\bf t}$ will contain infinitely many grid points as $n \rightarrow \infty$. Let us also assume that the population deepest point ${\bf m}_{c}$ is a continuous function on $I$. Then, under the conditions assumed in Fact \ref{fact2}, it is straightforward to show that $\sup_{{\bf t} \in I} |\widehat{m}_{{\bf t}} - m_{{\bf t}}| \rightarrow 0$ as $n \rightarrow \infty$ almost surely. \\
\indent We next consider the strong consistency of the empirical deepest point based on the spatial depth. As mentioned in Section \ref{1}, earlier work on the asymptotic consistency of the empirical spatial median in Hilbert spaces are restricted to showing convergence in the weak topology of the Hilbert space, i.e., convergence of $\langle{\bf h},\widehat{{\bf m}}_{s}\rangle$ to $\langle{\bf h},{\bf m}_{s}\rangle$ as $n \rightarrow \infty$ for each fixed ${\bf h}$ in the Hilbert space (see, e.g., \cite{Cadr01, Gerv08}). The following theorem asserts that the empirical spatial median in fact converges in the norm topology to its population version under very general conditions.
\begin{theoremm}  \label{thm5}
Let ${\bf X}$ be a random element in a strictly convex separable Hilbert space ${\cal X}$, and the distribution of ${\bf X}$ is nonatomic and not entirely supported on a line in ${\cal X}$. Also assume that for every $R > 0$, $\sup_{||{\bf x}|| \leq R} E\{||{\bf x} - {\bf X}||^{-1}\} < \infty$. Then, $||\widehat{{\bf m}}_{s} - {\bf m}_{s}|| \rightarrow 0$ as $n \rightarrow \infty$ almost surely.
\end{theoremm}
The assumption $\sup_{||{\bf x}|| \leq R} E\{||{\bf x} - {\bf X}||^{-1}\} < \infty$ in the above theorem holds under very general circumstances. Since ${\cal X}$ is a separable Hilbert space, any ${\bf x} \in {\cal X}$ can be written as ${\bf x} = \sum_{k=1}^{\infty} x_{k}\phi_{k}$ for an orthonormal basis $\{\phi_{k}\}_{k \geq 1}$ of ${\cal X}$. Let ${\bf X} = \sum_{k=1}^{\infty} X_{k}\phi_{k}$. If some two dimensional marginal of $(x_{1} - X_{1},x_{2} - X_{2},\ldots)$ has a density that is bounded on bounded subsets of $\mathbb{R}^{2}$, then the aforementioned assumption holds.
\begin{proof}[Proof of Theorem \ref{thm5}]
Let us first note that the assumptions in the statement of the theorem ensure that the spatial median is unique, and is the deepest point based on the spatial depth function (see Fact \ref{fact1} in Section \ref{2}). Recall that the spatial median ${\bf m}_{s}$ is the minimizer of the function $g({\bf x}) = E\{||{\bf x} - {\bf X}|| - ||{\bf X}||\}$. It follows from Lemma 2.1(i) in \cite{Cadr01} that $g(\widehat{{\bf m}}_{s}) \rightarrow g({\bf m}_{s})$ as $n \rightarrow \infty$ almost surely. Let us denote the Hessian of $g$ at ${\bf x}$ by $J_{{\bf x}} : {\cal X} \rightarrow {\cal X}$. It is a continuous linear operator on ${\cal X}$, and it is given by  
\begin{eqnarray*}
\langle J_{{\bf x}}({\bf z}),{\bf w}\rangle &=& E\left\{ \frac{\langle{\bf z},{\bf w}\rangle}{||{\bf x} - {\bf X}||} - \frac{\langle{\bf z},{\bf x} - {\bf X}\rangle \langle{\bf w},{\bf x} - {\bf X}\rangle}{||{\bf x} - {\bf X}||^{3}} \right\},
\end{eqnarray*}
where ${\bf z}, {\bf w} \in {\cal X}$. A second order Taylor expansion of the function $\Phi(v) = g({\bf m}_{s} + v{\bf h})$, where ${\bf h} = ({\bf x} - {\bf m}_{s})/||{\bf x} - {\bf m}_{s}||$, at $v = ||{\bf x} - {\bf m}_{s}||$ about $0$ yields
\begin{eqnarray}
g({\bf x}) - g({\bf m}_{s}) &=& E\left\{\frac{\langle{\bf m}_{s} - {\bf X},{\bf x} - {\bf m}_{s}\rangle}{||{\bf m}_{s} - {\bf X}||}\right\} + \frac{1}{2}\langle J_{{\bf y}}({\bf x} - {\bf m}_{s}),{\bf x} - {\bf m}_{s}\rangle,  \nonumber \\
&& \mbox{where $||{\bf y} - {\bf m}_{s}|| < ||{\bf x} - {\bf m}_{s}||$}  \label{eq5.2.1} \\
&=&  \frac{1}{2}\langle J_{{\bf y}}({\bf x} - {\bf m}_{s}),{\bf x} - {\bf m}_{s}\rangle.  \label{eq5.2.2} 
\end{eqnarray}
The last equality holds because $E\{({\bf m}_{s} - {\bf X})/||{\bf m}_{s} - {\bf X}||\} = {\bf 0}$, which follows from the nonatomicity of ${\bf X}$ and Theorem 4.14 in \cite{Kemp87}. It has been shown (see Proposition 2.1 in \cite{CCZ13}) that the function $g$ is strongly convex for all ${\bf x}$ in any closed and bounded ball around the origin. In other words, for each $A > 0$, there exists $c_{A} > 0$ such that
\begin{eqnarray}
 \langle J_{{\bf x}}({\bf h}),{\bf h}\rangle \geq c_{A},   \label{eq5.2.3}
\end{eqnarray}
for every $||{\bf x}|| \leq A$ and $||{\bf h}|| = 1$. This inequality is the key argument in the proof. In the finite dimensional setup, the above inequality is obtained using the compactness of the the sets $\{{\bf x} : ||{\bf x}|| \leq A\}$ and $\{{\bf h} : ||{\bf h}|| = 1\}$ along with the positive definiteness of the Hessian $J_{x}$. It is remarkable that even when we do not have the compactness of closed and bounded balls, inequality (\ref{eq5.2.3}) holds. Note that (\ref{eq5.2.3}) is equivalent to $\langle J_{{\bf x}}({\bf h}),{\bf h}\rangle \geq c_{A}||{\bf h}||^{2}$ for ${\bf h} \in {\cal X}$. Since for any $\epsilon > 0$, there exists $M > 0$ such that $P(||{\bf X} - {\bf m}_{s}|| > M) < \epsilon$, by Hoeffding's inequality we get that
\begin{eqnarray*}
\frac{1}{n} \sum_{i=1}^{n} I(||{\bf X}_{i} - {\bf m}_{s}|| > M) \leq 2\epsilon \ \mbox{as} \ n \rightarrow \infty \ \mbox{almost surely}.
\end{eqnarray*}
So, if $||{\bf x} - {\bf m}_{s}|| > 4M$, we have
\begin{eqnarray}
&& \frac{1}{n} \sum_{i=1}^{n} ||{\bf x} - {\bf X}_{i}||  \nonumber \\
&\geq& \frac{1}{n} \sum_{i=1}^{n} (||{\bf x} - {\bf m}_{s}|| - ||{\bf m}_{s} - {\bf X}_{i}||)I(||{\bf m}_{s} - {\bf X}_{i}|| \leq M) \nonumber \\
&& \ + \ \frac{1}{n} \sum_{i=1}^{n} (||{\bf m}_{s} - {\bf X}_{i}|| - ||{\bf x} - {\bf m}_{s}||)I(||{\bf m}_{s} - {\bf X}_{i}|| > M)  \nonumber \\
&>& \frac{(1-2\epsilon)}{2}||{\bf x} - {\bf m}_{s}|| + \frac{1}{n} \sum_{i=1}^{n} ||{\bf m}_{s} - {\bf X}_{i}|| I(||{\bf m}_{s} - {\bf X}_{i}|| \leq M)  \nonumber \\
&& \ + \frac{1}{n} \sum_{i=1}^{n} ||{\bf m}_{s} - {\bf X}_{i}|| I(||{\bf m}_{s} - {\bf X}_{i}|| > M) - 2\epsilon||{\bf x} - {\bf m}_{s}|| \nonumber \\
&=& \frac{(1-6\epsilon)}{2}||{\bf x} - {\bf m}_{s}|| + \frac{1}{n} \sum_{i=1}^{n} ||{\bf m}_{s} - {\bf X}_{i}|| \ > \ \frac{1}{n} \sum_{i=1}^{n} ||{\bf m}_{s} - {\bf X}_{i}||  \label{eq5.2.0}
\end{eqnarray}
if $\epsilon < 1/6$. Since $\widehat{{\bf m}}_{s}$ is the minimizer of $n^{-1} \sum_{i=1}^{n} ||{\bf x} - {\bf X}_{i}||$, there exists $R > 0$ such that $P(||\widehat{{\bf m}}_{s} - {\bf m}_{s}|| \leq R \ \mbox{as} \ n \rightarrow \infty) = 1$. This inequality along with (\ref{eq5.2.1}) and (\ref{eq5.2.3}) imply that
\begin{eqnarray}
\langle J_{{\bf y}}(\widehat{{\bf m}}_{s} - {\bf m}_{s}),\widehat{{\bf m}}_{s} - {\bf m}_{s}\rangle \geq c_{R}||\widehat{{\bf m}}_{s} - {\bf m}_{s}||^{2}   \label{eq5.2.4}
\end{eqnarray}
for some $c_{R} > 0$ as $n \rightarrow \infty$ almost surely. Since $g(\widehat{{\bf m}}_{s}) \rightarrow g({\bf m}_{s})$ as $n \rightarrow \infty$ almost surely, (\ref{eq5.2.2}) and (\ref{eq5.2.4}) together now imply that $||\widehat{{\bf m}}_{s} - {\bf m}_{s}|| \rightarrow 0$ as $n \rightarrow \infty$ almost surely. This completes the proof of the theorem.
\end{proof}
The strong convergence of the empirical spatial median in the norm topology is not restricted to separable Hilbert spaces only. For a class of Banach spaces, which include $L_{p}$ spaces for $1 < p < \infty$, it can be shown that the empirical spatial median converges in the norm topology to its population counterpart. As noted in the proof of the previous theorem, the key requirements there are the Frechet differentiability and the strong convexity of $g$ at ${\bf m}_{s}$. In other words, we need
\begin{eqnarray*}
g({\bf x}) - g({\bf m}_{s}) \geq g_{1}({\bf m}_{s})({\bf x} - {\bf m}_{s}) + \phi(||{\bf x} - {\bf m}_{s}||),
\end{eqnarray*}
where $g_{1}({\bf m}_{s}) \in {\cal X}^{*}$ is the Frechet derivative of $g$ at ${\bf m}_{s}$, and $\phi : [0,\infty) \rightarrow [0,\infty)$ is a convex lower semicontinuous function such that $\phi(0) = 0$ and $\phi(t) > 0$ if $t > 0$. Here ${\cal X}^{*}$ is the dual space of the Banach space ${\cal X}$. When ${\cal X}$ is a separable Hilbert space, $g_{1}({\bf x}) = E\left\{({\bf m}_{s} - {\bf X})/||{\bf m}_{s} - {\bf X}||\right\}$ and $\phi(t) = ct^{2}$ for some $c > 0$. The strong convexity of $g$ in a class of separable Banach spaces, which includes $L_{p}$ spaces for $1 < p < \infty$, follows from Proposition 1 and Theorem 3 in \cite{Aspl68}. Then, using similar arguments as in the proof of Theorem \ref{thm5}, we get that $||\widehat{{\bf m}}_{s} - {\bf m}_{s}|| \rightarrow 0$ as $n \rightarrow \infty$ almost surely in such spaces.

\bibliographystyle{model1b-num-names}
\bibliography{biblio-file.bib}

\begin{thebibliography}{25}
\expandafter\ifx\csname natexlab\endcsname\relax\def\natexlab#1{#1}\fi
\providecommand{\bibinfo}[2]{#2}
\ifx\xfnm\relax \def\xfnm[#1]{\unskip,\space#1}\fi
\bibitem[{Asplund(1968)}]{Aspl68}
\bibinfo{author}{E.~Asplund}, \bibinfo{title}{Fr\'echet differentiability of
  convex functions}, \bibinfo{journal}{Acta Math.} \bibinfo{volume}{121}
  (\bibinfo{year}{1968}) \bibinfo{pages}{31--47}.
\bibitem[{Bedall and Zimmermann(1979)}]{BZ79}
\bibinfo{author}{F.K. Bedall}, \bibinfo{author}{H.~Zimmermann},
  \bibinfo{title}{Algorithm {AS} 143: The mediancentre}, \bibinfo{journal}{J.
  Roy. Statist. Soc. Ser. C. Appl. Stat.} \bibinfo{volume}{28}
  (\bibinfo{year}{1979}) \bibinfo{pages}{325--328}.
\bibitem[{B{\"o}ttcher(2010)}]{Bott10}
\bibinfo{author}{B.~B{\"o}ttcher}, \bibinfo{title}{Feller processes: {T}he next
  generation in modeling. {B}rownian motion, {L}{\'e}vy processes and beyond},
  \bibinfo{journal}{PLoS ONE} \bibinfo{volume}{5} (\bibinfo{year}{2010})
  \bibinfo{pages}{e15102}.
\bibitem[{Brown(1983)}]{Brow83}
\bibinfo{author}{B.M. Brown}, \bibinfo{title}{Statistical uses of the spatial
  median}, \bibinfo{journal}{J. Roy. Statist. Soc. Ser. B} \bibinfo{volume}{45}
  (\bibinfo{year}{1983}) \bibinfo{pages}{25--30}.
\bibitem[{Cadre(2001)}]{Cadr01}
\bibinfo{author}{B.~Cadre}, \bibinfo{title}{Convergent estimators for the
  {$L_1$}-median of a {B}anach valued random variable},
  \bibinfo{journal}{Statistics} \bibinfo{volume}{35} (\bibinfo{year}{2001})
  \bibinfo{pages}{509--521}.
\bibitem[{Cardot et~al.(2013)Cardot, C{\'e}nac and Zitt}]{CCZ13}
\bibinfo{author}{H.~Cardot}, \bibinfo{author}{P.~C{\'e}nac},
  \bibinfo{author}{P.A. Zitt}, \bibinfo{title}{Efficient and fast estimation of
  the geometric median in hilbert spaces with an averaged stochastic gradient
  algorithm}, \bibinfo{journal}{Bernoulli} \bibinfo{volume}{19}
  (\bibinfo{year}{2013}) \bibinfo{pages}{18--43}.
\bibitem[{Chakraborty and Chaudhuri(2012)}]{CC12}
\bibinfo{author}{A.~Chakraborty}, \bibinfo{author}{P.~Chaudhuri},
  \bibinfo{title}{On data depth in infinite dimensional spaces},
  \bibinfo{year}{2012}. \bibinfo{note}{{T}echnical Report No. R4/2012,
  Theoretical Statistics and Mathematics Unit. Indian Statistical Institute,
  Kolkata, India}.
\bibitem[{Dhar and Chaudhuri(2011)}]{DC11}
\bibinfo{author}{S.S. Dhar}, \bibinfo{author}{P.~Chaudhuri}, \bibinfo{title}{On
  the statistical efficiency of robust estimators of multivariate location},
  \bibinfo{journal}{Stat. Methodol.} \bibinfo{volume}{8} (\bibinfo{year}{2011})
  \bibinfo{pages}{113--128}.
\bibitem[{Donoho and Gasko(1992)}]{DG92}
\bibinfo{author}{D.L. Donoho}, \bibinfo{author}{M.~Gasko},
  \bibinfo{title}{Breakdown properties of location estimates based on halfspace
  depth and projected outlyingness}, \bibinfo{journal}{Ann. Statist.}
  \bibinfo{volume}{20} (\bibinfo{year}{1992}) \bibinfo{pages}{1803--1827}.
\bibitem[{Fraiman and Muniz(2001)}]{FM01}
\bibinfo{author}{R.~Fraiman}, \bibinfo{author}{G.~Muniz},
  \bibinfo{title}{Trimmed means for functional data}, \bibinfo{journal}{Test}
  \bibinfo{volume}{10} (\bibinfo{year}{2001}) \bibinfo{pages}{419--440}.
\bibitem[{Gervini(2008)}]{Gerv08}
\bibinfo{author}{D.~Gervini}, \bibinfo{title}{Robust functional estimation
  using the median and spherical principal components},
  \bibinfo{journal}{Biometrika} \bibinfo{volume}{95} (\bibinfo{year}{2008})
  \bibinfo{pages}{587--600}.
\bibitem[{Kemperman(1987)}]{Kemp87}
\bibinfo{author}{J.H.B. Kemperman}, \bibinfo{title}{The median of a finite
  measure on a {B}anach space}, in: \bibinfo{booktitle}{Statistical data
  analysis based on the {$L_1$}-norm and related methods ({N}euch\^atel,
  1987)}, \bibinfo{publisher}{North-Holland}, \bibinfo{address}{Amsterdam},
  \bibinfo{year}{1987}, pp. \bibinfo{pages}{217--230}.
\bibitem[{Kosorok and Ma(2007)}]{KM07}
\bibinfo{author}{M.R. Kosorok}, \bibinfo{author}{S.~Ma},
  \bibinfo{title}{Marginal asymptotics for the ``large {$p$}, small {$n$}''
  paradigm: with applications to microarray data}, \bibinfo{journal}{Ann.
  Statist.} \bibinfo{volume}{35} (\bibinfo{year}{2007})
  \bibinfo{pages}{1456--1486}.
\bibitem[{Liu et~al.(1999)Liu, Parelius and Singh}]{LPS99}
\bibinfo{author}{R.Y. Liu}, \bibinfo{author}{J.M. Parelius},
  \bibinfo{author}{K.~Singh}, \bibinfo{title}{Multivariate analysis by data
  depth: descriptive statistics, graphics and inference},
  \bibinfo{journal}{Ann. Statist.} \bibinfo{volume}{27} (\bibinfo{year}{1999})
  \bibinfo{pages}{783--858}. \bibinfo{note}{With discussion and a rejoinder by
  Liu and Singh}.
\bibitem[{L{\'o}pez-Pintado and Romo(2009)}]{LPR09}
\bibinfo{author}{S.~L{\'o}pez-Pintado}, \bibinfo{author}{J.~Romo},
  \bibinfo{title}{On the concept of depth for functional data},
  \bibinfo{journal}{J. Amer. Statist. Assoc.} \bibinfo{volume}{104}
  (\bibinfo{year}{2009}) \bibinfo{pages}{718--734}.
\bibitem[{L{\'o}pez-Pintado and Romo(2011)}]{LPR11}
\bibinfo{author}{S.~L{\'o}pez-Pintado}, \bibinfo{author}{J.~Romo},
  \bibinfo{title}{A half-region depth for functional data},
  \bibinfo{journal}{Comput. Statist. Data Anal.} \bibinfo{volume}{55}
  (\bibinfo{year}{2011}) \bibinfo{pages}{1679--1695}.
\bibitem[{Maguluri and Singh(1997)}]{MS97}
\bibinfo{author}{G.~Maguluri}, \bibinfo{author}{K.~Singh}, \bibinfo{title}{On
  the fundamentals of data robustness}, in: \bibinfo{booktitle}{Robust
  inference}, volume~\bibinfo{volume}{15} of \textit{\bibinfo{series}{Handbook
  of Statist.}}, \bibinfo{publisher}{North-Holland},
  \bibinfo{address}{Amsterdam}, \bibinfo{year}{1997}, pp.
  \bibinfo{pages}{537--549}.
\bibitem[{Revuz and Yor(1991)}]{RY91}
\bibinfo{author}{D.~Revuz}, \bibinfo{author}{M.~Yor},
  \bibinfo{title}{Continuous martingales and {B}rownian motion}, volume
  \bibinfo{volume}{293} of \textit{\bibinfo{series}{Grundlehren der
  Mathematischen Wissenschaften [Fundamental Principles of Mathematical
  Sciences]}}, \bibinfo{publisher}{Springer-Verlag}, \bibinfo{address}{Berlin},
  \bibinfo{year}{1991}.
\bibitem[{Serfling(2002)}]{Serf02}
\bibinfo{author}{R.~Serfling}, \bibinfo{title}{A depth function and a scale
  curve based on spatial quantiles}, in: \bibinfo{booktitle}{Statistical data
  analysis based on the {$L_1$}-norm and related methods ({N}euch\^atel,
  2002)}, Stat. Ind. Technol., \bibinfo{publisher}{Birkh\"auser},
  \bibinfo{address}{Basel}, \bibinfo{year}{2002}, pp. \bibinfo{pages}{25--38}.
\bibitem[{Singh(1993)}]{Sing93}
\bibinfo{author}{K.~Singh}, \bibinfo{title}{Paradoxes in robustness}, in:
  \bibinfo{booktitle}{Statistics and Probability: A Raghu Raj Bahadur
  Festschrift, ed. J. K. Ghosh, S. K. Mitra, K. R. Parthasarathy and B. L. S.
  Prakasa Rao}, \bibinfo{publisher}{Wiley Eastern Limited Publishers},
  \bibinfo{year}{1993}, pp. \bibinfo{pages}{531--537}.
\bibitem[{Small(1990)}]{Smal90}
\bibinfo{author}{C.G. Small}, \bibinfo{title}{A survey of multidimensional
  medians}, \bibinfo{journal}{Int. Statist. Rev.} \bibinfo{volume}{58}
  (\bibinfo{year}{1990}) \bibinfo{pages}{263--277}.
\bibitem[{Sun and Genton(2011)}]{SG11}
\bibinfo{author}{Y.~Sun}, \bibinfo{author}{M.G. Genton},
  \bibinfo{title}{Functional boxplots}, \bibinfo{journal}{J. Comput. Graph.
  Statist.} \bibinfo{volume}{20} (\bibinfo{year}{2011})
  \bibinfo{pages}{316--334}.
\bibitem[{Valadier(1984)}]{Vala84}
\bibinfo{author}{M.~Valadier}, \bibinfo{title}{La multi-application m\'edianes
  conditionnelles}, \bibinfo{journal}{Z. Wahrsch. Verw. Gebiete}
  \bibinfo{volume}{67} (\bibinfo{year}{1984}) \bibinfo{pages}{279--282}.
\bibitem[{Vardi and Zhang(2000)}]{VZ00}
\bibinfo{author}{Y.~Vardi}, \bibinfo{author}{C.H. Zhang}, \bibinfo{title}{The
  multivariate {$L_1$}-median and associated data depth},
  \bibinfo{journal}{Proc. Natl. Acad. Sci. USA} \bibinfo{volume}{97}
  (\bibinfo{year}{2000}) \bibinfo{pages}{1423--1426 (electronic)}.
\bibitem[{Zuo and Serfling(2000)}]{ZS00a}
\bibinfo{author}{Y.~Zuo}, \bibinfo{author}{R.~Serfling},
  \bibinfo{title}{General notions of statistical depth function},
  \bibinfo{journal}{Ann. Statist.} \bibinfo{volume}{28} (\bibinfo{year}{2000})
  \bibinfo{pages}{461--482}.

\end{thebibliography}

\end{document}